\documentclass[a4paper,11pt,oneside]{article}
\usepackage{latexsym}
\usepackage{amssymb,amsfonts,amsmath,mathrsfs}

\usepackage{amssymb}
\usepackage{color}
\usepackage[usenames,dvipsnames]{xcolor}
\usepackage[colorlinks=true,linkcolor=Blue,citecolor=OliveGreen]{hyperref}
\usepackage{amsfonts}
\usepackage{amsthm}
\usepackage{amsmath}
\usepackage{amscd}
\usepackage{geometry}
\usepackage{mathrsfs}
\usepackage{tikz}
\usepackage{caption}
\usepackage{pgfplots}
\usepackage{authblk}
\usepackage{float}

\newtheorem{prop}{Proposition}
\newtheorem{thm}{Theorem}
\newtheorem{lem}{Lemma}

\theoremstyle{definition}
\newtheorem*{rem}{Remark}

\newcommand{\ve}{{\varepsilon}}
\newcommand{\vt}{{\vartheta}}
\newcommand{\ml}{{\mathcal L}}
\newcommand{\ul}{\underline}

\title{Gaps between zeros of Dedekind zeta-functions\\ of quadratic number fields. II}

\author{H. M. Bui}
\affil{School of Mathematics, University of Manchester, Manchester M13 9PL, UK, hung.bui@manchester.ac.uk}

\author{Winston P. Heap}
\affil{Department of Mathematical Sciences, Norwegian University of Science and Technology, 
NO-7491 Trondheim, Norway, winstonheap@gmail.com}

\author{Caroline L. Turnage-Butterbaugh}
\affil{Department of Mathematics, Duke University, Durham, NC 27708 USA, cturnagebutterbaugh@gmail.com}

\date{}

\begin{document}
\allowdisplaybreaks
\maketitle

\begin{abstract}
Let $K$ be a quadratic number field and $\zeta_K(s)$ be the associated Dedekind zeta-function. We show that there are infinitely many gaps between consecutive zeros of $\zeta_K(s)$ on the critical line which are greater than $2.866$ times the average spacing. 
\end{abstract}

\section{Introduction}

Let $K$ be a number field, and let $\mathcal{O}_K$ denote its ring of integers. The Dedekind zeta-function attached to $K$ is defined in the half-plane $\Re(s)>1$ by
\[
\zeta_K(s)= \sum_{\mathfrak{a}\subset \mathcal{O}_K}\frac{1}{N(\mathfrak{a})^s} = \prod_{\mathfrak{p}\subset \mathcal{O}_K}\left(1-\frac{1}{N(\mathfrak{p})^s}\right)^{-1},
\]
where $\mathfrak{a}$ and $\mathfrak{p}$ run over the nonzero ideals and primes ideals of $\mathcal{O}_K$, respectively. Let $K$ be a quadratic extension of $\mathbb{Q}$ with discriminant $D$. It is known that $\zeta_K(s)$ factors as
\[
\zeta_K(s) = \zeta(s)L(s,\chi_D),
\]
where $\zeta(s)$ is the Riemann zeta-function and $L(s,\chi_D)$ is the Dirichlet $L$-function associated to $\chi_D$. We denote the modulus of $\chi_D$ by $q$, which is given by the formula
\begin{equation*}
q=\begin{cases}4|D| &\text{if}\,\,D\equiv 2 (\!\!\!\!\!\!\mod 4), \\ |D|&\text{otherwise.}
\end{cases}\end{equation*}

Let $0\leq\gamma_1\leq\gamma_2\leq\ldots\leq\gamma_n\leq\ldots$ denote the imaginary parts of the nontrivial zeros of $\zeta_K(s)$ in the upper half-plane. Also, let $t_n$ be the imaginary part of the $n$-th zero of $\zeta_K(s)$ of the form $\rho=1/2+it$, $t>0$. In this article, we study the vertical distribution between the nontrivial zeros of $\zeta_K(s)$. For $T\ge 2$, it is well-known that
\begin{equation*}
N_K(T):= \sum_{0 < \gamma \le T}1 = \frac{T\ml}{\pi}- \frac{T}{\pi} + O\left(\ml\right),
\end{equation*}
where
\[
\ml=\log\frac{\sqrt{|D|}T}{4\pi^2}.
\]
Hence the average size of $\gamma_{n+1}- \gamma_{n}$ is $\pi / \log(\sqrt{|D|}\gamma_n)$. Normalizing, let
\[
\lambda_K := \limsup_{n\to \infty} \frac{\gamma_{n+1}- \gamma_{n}}{\pi / \log(\sqrt{|D|}\gamma_n)} \qquad \text{and} \qquad \Lambda_K := \limsup_{n\to \infty} \frac{t_{n+1}- t_{n}}{\pi / \log(\sqrt{|D|}t_n)}.
\]
Note that the Generalized Riemann Hypothesis for $\zeta_K(s)$ implies $\lambda_K \!=\! \Lambda_K$. By definition, we have that $\Lambda_K \!\ge\! \lambda_K \!\ge\! 1$, but it is believed that $\Lambda_K \!=\! \lambda_K \!=\! \infty$. Using a method of Hall \cite{Hall} and some ideas of Bredberg \cite{Bredberg}, Turnage-Butterbaugh \cite{CTB} has shown that $\Lambda_K \ge \sqrt{6}=2.449\ldots$. Our main theorem is an improvement upon this bound on $\Lambda_K$.

\begin{thm}\label{thm:main}
We have $\Lambda_K > 2.866$. Thus, under the Generalized Riemann Hypothesis for $\zeta_K(s)$, we have $\lambda_K > 2.866$. \end{thm}

\begin{rem}
Theorem 1 of \cite{CTB} contains a small error, but the proof and result $\Lambda_K \ge2.449\ldots$ are valid after making an appropriate modification, which we now explain. The theorem states that for $\varepsilon\!>\!0$ and $|D|\!<\!T^{7/9-\varepsilon}$, we have $\Lambda_K \ge2.449\ldots$, however, we must require $D \ll\! T^{\varepsilon}$ for the result $\Lambda_K \ge2.449\ldots$ to hold. The mistake arises in the third line of the proof of Theorem 3 of \cite{CTB}, where it is stated that for $\alpha\!+\!\beta \!\ll\! 1/\log(\sqrt{|D|}T)$, we have $D^{-(\alpha\!+\!\beta)} \!=\! 1 + O((\alpha\!+\!\beta)|D|^\epsilon)$. The correct estimate is, rather,  $D^{-(\alpha+\beta)} = 1 + O((\alpha\!+\!\beta)\log |D|)$. Thus we require  $D\ll T^{\varepsilon}$ so that the $O$-term is smaller than the main term. By inserting the correct estimate for $D^{-(\alpha+\beta)}$ in the argument, the proof is valid. In short, the statement of Theorem 1 of \cite{CTB} should read as follows: for $\varepsilon\!>\!0$ and $D\!\ll\! T^{\varepsilon}$, we have $\Lambda_K \ge2.449\ldots$. We note that Theorem \ref{thm:main} above holds for $D\ll T^{\varepsilon}$ as well.
\end{rem} 

The weaker lower bound on $\Lambda_K$ obtained in \cite{CTB} is derived using the mixed second moments of the derivatives of $\zeta_K(1/2+it)$. Those moments are obtained via a special case of a result of Heap  \cite[Theorem 1]{Heap} on the twisted second moment of the Dedekind zeta-function. Combining Hall's method with the full strength of Heap's result would likely improve the bound derived in \cite{CTB}. In the present article, however, we appeal to a recent improvement of Heap's result due to Bettin {\it et al.} \cite{BBLR} which allows one to consider a longer amplifier. Indeed, in \cite[Corollary 1]{Heap}, the Dirichlet polynomial has length  $T^{\vartheta},\, \vartheta \!<\! 1/11$, and in \cite{BBLR}, one may take the length of the Dirichlet polynomial to be $T^{\vartheta},\, \vartheta \!<\! 1/4.$ 

The analogous problem for demonstrating large gaps relative to the average spacing between critical zeros of the Riemann zeta-function, $\zeta(s)$, has been extensively studied (see, for example, \cite{Bredberg, Bui, Bui2, HBgaps, BMN, CGG, CGG2, FW, Hall1, Hall, MO, Mu, Ng}). To place our result in context, we note that, to date, the strongest result for $\zeta(s)$ is due to Bui and Milinovich \cite{HBgaps}, who have shown, assuming the Riemann Hypothesis, that there are infinitely many critical zeros of $\zeta(s)$ which are at least 3.18 times the average spacing. Very recently the problem has also been studied for critical zeros of a GL(2) $L$-function, $L(s,f)$, where $f$ may be taken as either a primitive (holomorphic or Maass) cusp form. Barrett {\it et al.} \cite{Betal}, assuming the Generalized Riemann Hypothesis, have shown that there are infinitely many critical zeros of $L(s,f)$ which are at least $\sqrt{3}=1.732\ldots$ times the average spacing. Both of these results make use of Hall's method, which we detail in the next section.

\section{A lower bound on $\Lambda_K$}
In this section, we outline Hall's method to obtain the bound appearing in Theorem \ref{thm:main}. We use a similar setup as Bui and Milinovich \cite{HBgaps} in the construction of our argument, which in turn follows Section 2 of Bredberg \cite{Bredberg}.

\begin{thm}[Wirtinger's inequality]\label{wirtinger} Let $f:[a,b]\to \mathbb{C}$ be a continuously differentiable function, and suppose that  $f(a)\!=\!f(b)\!=\!0$. Then
\[\int_a^b|f(x)|^2dx\leq \Big(\frac{b-a}{\pi}\Big)^2\int_a^b|f^\prime(x)|^2 dx.\]
\end{thm}
\begin{proof}
This is a variation of Wirtinger's inequality given by Bredberg (see \cite[Corollary 1]{Bredberg}).
\end{proof}
\smallskip

Fix $K$. Suppose, towards a contradiction, that $\Lambda_K \le \kappa$, where $\kappa$ is some real number. Let
\[
f(t):=e^{i\nu\ml t}\zeta(\tfrac{1}{2}\!+\!it)L(\tfrac{1}{2}\!+\!it,\chi_D)M(\tfrac{1}{2}\!+\!it),
\]
where $\nu$ is a real constant that will be chosen later. Here we choose $M(s)$ to be an amplifier of the form
\[
M(s)=\sum_{h_1h_2\leq y}\frac{d_r(h_1)d_r(h_2)\chi_D(h_2)P[h_1h_2]}{(h_1h_2)^s},
\]
where $y=T^{\vartheta},\, 0 < \vartheta < 1/4,\, r \in \mathbb{N},$ and $d_r(h)$ denotes the coefficients of $\zeta(s)^r$. Moreover, we use the notation
\[
P[h]=P\Big(\frac{\log y/h}{\log y}\Big)
\]
for $1 \le h \le y$, where $P(x) = \sum_{j\ge 0}b_jx^j$ is a polynomial. By convention, we set $P[h]= 0$ for $h \ge y$. Thus, for all $h$,  we have
\begin{equation}\label{P integral}
P[h]=\sum_{j\geq 0}\frac{b_j j!}{(\log y)^j}\frac{1}{2\pi i}\int_{(1)}\Big(\frac{y}{n}\Big)^s \frac{ds}{s^{j+1}}.
\end{equation}
Here, and throughout the article, the notation $\int_{(c)}$ means $\int_{c-i\infty}^{c+i\infty}.$

Denote all of the zeros of the function $f(t)$ in the interval $[T,2T]$ by $\widetilde{t}_1 \le \widetilde{t}_2 \le \ldots \le \widetilde{t}_N$. By our assumption, we have
\[
\widetilde{t}_{n+1}-\widetilde{t}_{n} \le \big(1+o(1)\big)\frac{\kappa \pi}{\ml}
\]
as $T \to \infty$ for $n=0,1,\ldots,N$. Here we define $\widetilde{t}_0=T$ and $\widetilde{t}_{N+1}=2T$. By Theorem \ref{wirtinger}, we have
\begin{equation}\label{eq:bywirtinger}
\int_{\widetilde{t}_n}^{\widetilde{t}_{n+1}}|f(t)|^2dt \leq \big(1+o(1)\big)\frac{\kappa^2}{\ml^2}\int_{\widetilde{t}_n}^{\widetilde{t}_{n+1}}|f^\prime(t)|^2dt
\end{equation}
for $n=0,1,\ldots,N$. Summing \eqref{eq:bywirtinger} for all zeros in the range $[T,2T]$ gives
\[
\int_T^{2T}|f(t)|^2dt\leq \big(1+o(1)\big)\frac{\kappa^2}{\ml^2}\int_T^{2T}|f^\prime(t)|^2dt.
\]
Therefore, if
\[
h(D,\kappa,M):=\frac{\ml^2}{\kappa^2}\frac{\int_T^{2T}|f(t)|^2dt}{\int_T^{2T}|f^\prime(t)|^2dt} > 1,
\]
we may conclude that $\Lambda_K > \kappa$. Thus, the problem has been reduced to the study of the above mean values.

\subsection{Smoothed mean-value estimates}
To calculate the mean squares of $|f(t)|$ and $|f^\prime(t)|$ we consider the more general integral
\begin{multline}\label{Idef}
I(\underline{\alpha}, \underline{\beta}):=\int_{-\infty}^\infty \zeta(\tfrac{1}{2}\!+\!\alpha_1\!+\!it)L(\tfrac{1}{2}\!+\!\alpha_2\!+\!it,\chi_D)\zeta(\tfrac{1}{2}\!+\!\beta_1\!-\!it)
L(\tfrac{1}{2}\!+\!\beta_2\!-\!it,\chi_D)\\\times M(\tfrac{1}{2}\!+\!\alpha_3\!+\!it)M(\tfrac{1}{2}\!+\!\beta_3-it)\Phi\Big(\frac tT\Big)dt,
\end{multline}
where $\alpha_i, \beta_i \ll \ml^{-1}$ for $i=1,2,3$, $y=T^{\vt},\, \vt<1/4,$ and $\Phi(x)$ is a smooth function supported in $[1, 2]$ with derivatives
$\Phi^{(j)}(x) \ll_j T^\varepsilon$ for any $j\geq 0$. We shall prove the following proposition in Section \ref{propproof}.

\begin{prop}\label{prop:key}We have
\[
I(\ul{\alpha},\ul{\beta})=\frac{c(\ul{\alpha},\ul{\beta})A_{r+1}(\log y)^{2r^2+4r}\ml^2}{(2r^2-1)!((r-1)!)^4}\widehat{\Phi}(0)T+O(T\ml^{2r^2+4r+1}),
\]
where 
\[A_r=\prod_p\left(1-\frac{1}{p}\right)^{2r^2}\sum_{m=0}^\infty\frac{a_r(p^m)^2}{p^m},\]
with $a_r(n)$ the coefficients of $\zeta_K(s)^r$, and 
\begin{align*}
c(\ul{\alpha},\ul{\beta})&=\int\limits_{\substack{[0,1]^7\\x+x_1+x_2\leq 1\\x+x_3+x_4\leq 1}}y^{-(\alpha_3+\beta_3)x-\alpha_3(x_1+x_2)-\beta_3(x_3+x_4)-\beta_1x_1-\beta_2x_2-\alpha_1x_3-\alpha_2x_4}\notag\\
&\qquad\times[Ty^{-(x_1+x_3)}]^{-(\alpha_1+\beta_1)t_1}[qTy^{-(x_2+x_4)}]^{-(\alpha_2+\beta_2)t_2}\big(1\!-\!\vt(x_1\!+\!x_3)\big)\big(1\!-\!\vt(x_2\!+\!x_4)\big)\\
&\qquad\qquad\times x^{2r^2-1}(x_1x_2x_3x_4)^{r-1}P(1\!-\!x\!-\!x_1\!-\!x_2)P(1\!-\!x\!-\!x_3\!-\!x_4)dxdx_1dx_2dx_3dx_4dt_1dt_2.
\end{align*}
\end{prop}

The proof of Proposition \ref{prop:key} uses the main theorem of \cite{BBLR} (see Theorem \ref{twists} below). This gives $I(\ul{\alpha},\ul{\beta})$ as a sum of six terms, two of which lead to a  lower order contribution. The difference between the four main terms is simply a permutation of the shifts which  allows us to concentrate on a single term.  Using standard analytic techniques we compute this as a product of $(\log y)^{2r^2+4r}$ times a multiple integral dependent on the shifts. Upon combining all four main terms we gain the full expression for $c(\ul{\alpha},\ul{\beta})$ along with the extra factor of $\ml^2$.

Using Proposition \ref{prop:key}, we now prove the following smoothed mean-values of $f(t)$ and $f^\prime(t)$.

\begin{prop}\label{thm:smoothsquare}
Suppose $\vartheta < 1/4$. Then we have
\[
\int_{-\infty}^{\infty}|f(t)|^2\Phi\Big(\frac tT\Big)dt = \frac{c_0A_{r+1}(\log y)^{2r^2+4r}\ml^2}{(2r^2-1)!((r-1)!)^4}\widehat{\Phi}(0)T + O(T\ml^{2r^2+4r+1}),
\]
where
\begin{multline}
c_0:= c(\ul{0},\ul{0})=\int\limits_{\substack{[0,1]^7\\x+x_1+x_2\leq 1\\x+x_3+x_4\leq 1}}\big(1\!-\!\vt(x_1\!+\!x_3)\big)\big(1\!-\!\vt(x_2\!+\!x_4)\big)x^{2r^2-1}(x_1x_2x_3x_4)^{r-1}\notag\\
\times P(1\!-\!x\!-\!x_1\!-\!x_2)P(1\!-\!x\!-\!x_3\!-\!x_4)dxdx_1dx_2dx_3dx_4dt_1dt_2.
\end{multline}
\end{prop}

\begin{proof}
The result follows immediately from Proposition \ref{prop:key} upon noting that
\[
\int_{-\infty}^{\infty} |f(t)|^2\Phi\Big(\frac tT\Big)dt=I(\underline{0},\underline{0}).
\]
\end{proof}

\begin{prop}\label{thm:smoothprimesquare}
Suppose $\vartheta < 1/4$. Then we have
\[
\int_{-\infty}^{\infty}|f^{\prime}(t)|^2\Phi\Big(\frac tT\Big)dt = \frac{c_1(\nu)A_{r+1}(\log y)^{2r^2+4r}\ml^4}{(2r^2-1)!((r-1)!)^4}\widehat{\Phi}(0)T + O(T\ml^{2r^2+4r+3}),
\]
where
\begin{align*}
c_1(\nu) :=&\int\limits_{\substack{[0,1]^7\\x+x_1+x_2\leq 1\\x+x_3+x_4\leq 1}}\big(1\!-\!\vt(x_1\!+\!x_3)\big)\big(1\!-\!\vt(x_2\!+\!x_4)\big)x^{2r^2-1}(x_1x_2x_3x_4)^{r-1}\notag\\
&\qquad\times\Big(\nu \!-\! \vartheta(x+x_1\!+\!x_2\!+\!x_3\!+\!x_4)-t_1\big(1\!-\!\vartheta(x_1\!+\!x_3)\big)-t_2\big(1\!-\!\vartheta(x_2\!+\!x_4)\big)\Big)^2\\
&\qquad\qquad\quad\times P(1\!-\!x\!-\!x_1\!-\!x_2)P(1\!-\!x\!-\!x_3\!-\!x_4)dxdx_1dx_2dx_3dx_4dt_1dt_2.
\end{align*}
\end{prop}

\begin{proof}
We have
\begin{align*}
\int_{-\infty}^{\infty} |f^\prime(t)|^2\Phi\Big(\frac tT\Big)dt=&\ml^2Q\bigg[\frac{1}{\ml}\bigg(\frac{d}{d\alpha_1}+\frac{d}{d\alpha_2}+\frac{d}{d\alpha_3}\bigg)\bigg]\\
&\qquad\qquad\times Q\bigg[\frac{1}{\ml}\bigg(\frac{d}{d\beta_1}+\frac{d}{d\beta_2}+\frac{d}{d\beta_3}\bigg)\bigg]I(\underline{\alpha}, \underline{\beta})\bigg|_{\underline{\alpha}=\underline{\beta}=\underline{0}},
\end{align*}
where
\[
Q(x)=\nu+x.
\] 
Note that
\[
Q\bigg[\frac{1}{\ml}\bigg(\frac{d}{d\alpha_1}+\frac{d}{d\alpha_2}+\frac{d}{d\alpha_3}\bigg)\bigg]X_1^{\alpha_1}X_2^{\alpha_2}X_3^{\alpha_3}=Q\bigg[\frac{\log X_1+\log X_2+\log X_3}{\ml}\bigg]X_1^{\alpha_1}X_2^{\alpha_2}X_3^{\alpha_3}.
\]
Since the error term in $I(\ul{\alpha},\ul{\beta})$ is necessarily holomorphic for small $\alpha_i,\beta_j$, we may apply the above differential operator via Cauchy's integral formula over a circle of radius $\ml^{-1}$ to give the desired formula for  $c_1(\nu)$ along with the desired error term. 
\end{proof}

\subsection{Completion of the proof of Theorem \ref{thm:main}}
From Propositions \ref{thm:smoothsquare} and \ref{thm:smoothprimesquare}, we deduce that
\begin{equation}\label{eq:meansquaref}
\int_{T}^{2T}|f(t)|^2dt = \frac{c_0A_{r+1}(\log y)^{2r^2+4r}\ml^2}{(2r^2-1)!((r-1)!)^4}T + O(T\ml^{2r^2+4r+1})
\end{equation}
and
\begin{equation}\label{eq:meansquarefprime}
\int_{T}^{2T}|f^{\prime}(t)|^2dt = \frac{c_1(\nu)A_{r+1}(\log y)^{2r^2+4r}\ml^4}{(2r^2-1)!((r-1)!)^4}T+ O(T\ml^{2r^2+4r+3}).
\end{equation}
By \eqref{eq:meansquaref} and \eqref{eq:meansquarefprime}, we have
\begin{align*}
h(D,\kappa,M) &:=\frac{\ml^2}{\kappa^2}\frac{\int_T^{2T}|f(t)|^2dt}{\int_T^{2T}|f^\prime(t)|^2dt}\\
& = \frac{c_0}{\kappa^2c_1(\nu)} + o(1).
\end{align*}
The choice $\vartheta=1/4$, $v=1.2773$, $r=1$, and
\[
P(x)=1-10.8998x+28.9444x^2-22.1343x^3+0.6148x^4
\]
gives
\[
h(D,2.866,M)=1.00016\ldots,
\]
and the theorem follows.

\begin{rem}It is not clear that emulating higher moments should necessarily lead to larger gaps in our framework. The basic point is that the coefficients in the denominator of the ratio $\int |f|^2/\int |f^\prime|^2$ can often be larger than those of the numerator when one considers higher moments. 

For example, taking $f(t)=\zeta(1/2+it)^k$ in our setup leads to lower bounds on $\Lambda_\mathbb{Q}^2$ of (essentially) the form
\begin{equation}\label{ratio}
\big(\tfrac{1}{2}\log T\big)^2\frac{\int_T^{2T}|\zeta(\tfrac{1}{2}+it)|^{2k}dt}{\int_T^{2T}|k\zeta^\prime(\tfrac{1}{2}+it)\zeta(\tfrac{1}{2}+it)^{k-1}|^2dt}.
\end{equation}
If we assume the moments conjecture of Keating-Snaith \cite{KS},
\[\int_T^{2T}|\zeta(\tfrac{1}{2}+it)|^{2k}dt\sim \frac{a(k)g(k)}{k^2!}T(\log T)^{k^2},\]
along with the following conjecture of Hughes \cite[Equation (6.81)]{Hughes thesis},
\[\int_T^{2T}|\zeta^\prime(\tfrac{1}{2}+it)\zeta(\tfrac{1}{2}+it)^{k-1}|^2dt\sim \frac{a(k)g(k)}{k^2!}\frac{k^2}{4k^2-1}T(\log T)^{k^2+2},\]
then the ratio in \eqref{ratio} is given by $(4k^2-1)/4k^4$ which tends to zero as $k\to\infty$.  By considering Hardy's $Z$-function instead of the zeta-function one can make a slight improvement here. Even then, assuming a conjecture of Hughes \cite[Conjecture 6.1]{Hughes thesis} the best one can get from this method is $\Lambda_\mathbb{Q}\geqslant 2$.

As we have demonstrated, one can perform further optimizations in this setup. However, when considering higher moments (or their emulation via twisted moments) there may be more germane inequalities other than the basic inequality of Wirtinger. This is the subject of Hall's paper \cite{Hall2}, where unfortunately the resulting lower bounds on $\Lambda_\mathbb{Q}$ are seemingly very difficult to calculate for $k\geq 4$.  

If one is to assume the higher moment conjectures then a better approach is to use the method of Mueller \cite{Mu}. Briefly, the reason for this is that there is an issue of overcounting in Mueller's method, but this can be reduced by taking higher moments and this allows one to detect larger gaps. For example, in the paper \cite{SS} it is proved that there are infinitely large gaps subject to the appropriate moment conjectures.
\end{rem}

\section{Proof of Proposition \ref{prop:key}}\label{propproof}
\subsection{Lemmas}
We first collect some lemmas of Bui and Milinovich \cite[Lemma 4.1 and Lemma 4.2]{HBgaps}.

\begin{lem}\label{K lem}Let
\[
K_j(\alpha, \beta)= \frac{1}{2\pi i}\int_{(\ml^{-1})}\Big( \frac{y}{n}\Big)^u\zeta^{r}(1\!+\!\alpha\!+\!u)\zeta^{r}(1\!+\!\beta\!+\!u)\frac{du}{u^{j+1}}.
\]
Then we have 
\begin{align*}
K_{j}(\alpha, \beta)=\frac{(\log y/n)^{j+2r}}{((r-1)!)^2j!}\iint\limits_{\substack{x_1+x_2\leq 1\\0\leq x_1,x_2\leq 1}}\Big(\frac{y}{n}\Big)^{-\alpha x_1-\beta x_2}(x_1x_2)^{r-1}(1-x_1-x_2)^jdx_1dx_2+O(\ml^{j+2r-1})
\end{align*}
uniformly for $\alpha, \beta \ll \ml^{-1}$.
\end{lem}

\begin{lem}\label{d sum lem}Suppose $f$ is a smooth function. Then we have
\begin{equation*}
\sum_{n\leq y}\frac{d_r(n)}{n^{1+\alpha}}f\Big(\frac{\log y/n}{\log y}\Big)=\frac{\log^r y}{(r-1)!}\int_0^1y^{-\alpha x}x^{r-1}f(1-x)dx+O(\ml^{r-1}).
\end{equation*}
\end{lem}

\subsection{Mean value calculations}

To evaluate $I(\ul{\alpha},\ul{\beta})$, defined in \eqref{Idef}, we apply a special case of a result of~\cite{BBLR}, which computes the twisted moment of the product of four Dirichlet $L$-functions. 

Let $G(s)$ be an even entire function of rapid decay in any fixed strip $|\Re(s)|\leq C$ satisfying $G(0)=1$. Furthermore, assume that $G(s)$ vanishes at $s=-(\alpha_i+\beta_i)/2$ for $i=1, 2$. Let
\begin{equation}\label{Vx}
V(x)=\frac{1}{2\pi i}\int_{(1)}G(s)(2\pi)^{-2s}x^{-s}\frac{ds}{s}.
\end{equation}
Given two positive integers $a$ and $b$, and four primitive characters $\psi_1, \psi_2, \chi_1,$ and $\chi_2,$ we denote  
\[Z_{\alpha,\beta,\gamma,\delta,a,b}(t,\psi_1,\psi_2,\chi_1,\chi_2):=\sum_{am_1m_2=bn_1n_2}\frac{\psi_1(m_1)\psi_2(m_2)\chi_1(n_1)\chi_2(n_2)}{m_1^{1/2+\alpha}m_2^{1/2+\beta}n_1^{1/2+\gamma}n_2^{1/2+\delta}}V\Big(\frac{m_1m_2n_1n_2}{t^2}\Big).\]

Since $\zeta_K(s)=\zeta(s)L(s,\chi_D)$, we shall take $\psi_1\!=\!\chi_1\!=\!1$, the trivial character, and $\psi_2\!=\!\chi_2\!=\!\chi_D$. Applying Theorem 1.3 of \cite{BBLR} in this special case, we obtain the following theorem.

\begin{thm}\label{twists}Suppose $\vt<1/4$. Then we have
\begin{align*}
&\int_{-\infty}^\infty \zeta(\tfrac{1}{2}\!+\!\alpha_1\!+\!it)L(\tfrac{1}{2}\!+\!\alpha_2\!+\!it,\chi_D)\zeta(\tfrac{1}{2}\!+\!\beta_1\!-\!it)
L(\tfrac{1}{2}\!+\!\beta_2\!-\!it,\chi_D)\sum_{h,k\leq y}\frac{c_hc_k}{h^{1/2+it}k^{1/2-it}}\Phi\Big(\frac tT\Big)dt\nonumber\\
&\qquad=\sum_{h,k\leq y}\frac{c_hc_k}{\sqrt{hk}}\int_{-\infty}^\infty\Phi\Big(\frac tT\Big)\bigg[Z_{\alpha_1,\alpha_2,\beta_1,\beta_2,h,k}(t,1,\chi_D,1,\chi_D)\nonumber\\
&\qquad\qquad+\Big(\frac{t}{2\pi}\Big)^{-(\alpha_1+\beta_1)}Z_{-\beta_1,\alpha_2,-\alpha_1,\beta_2,h,k}(t,1,\chi_D,1,\chi_D)\nonumber\\
&\qquad\qquad+\Big(\frac{qt}{2\pi}\Big)^{-(\alpha_2+\beta_2)}Z_{\alpha_1,-\beta_2,\beta_1,-\alpha_2,h,k}(t,1,\chi_D,1,\chi_D)\\
&\qquad\qquad+\Big(\frac{t}{2\pi}\Big)^{-(\alpha_1+\beta_1)}\Big(\frac{qt}{2\pi}\Big)^{-(\alpha_2+\beta_2)}Z_{-\beta_1,-\beta_2,-\alpha_1,-\alpha_2,h,k}(t,1,\chi_D,1,\chi_D)\nonumber\\
&\qquad\qquad+\ve_{\chi_D}\Big(\frac{t}{2\pi}\Big)^{-\alpha_1}\Big(\frac{qt}{2\pi}\Big)^{-\beta_2}Z_{-\beta_2,\alpha_2,\beta_1,-\alpha_1,h,qk}(t,\chi_D,\chi_D,1,1)\nonumber\\
&\qquad\qquad+\ve_{\chi_D}\Big(\frac{t}{2\pi}\Big)^{-\beta_1}\Big(\frac{qt}{2\pi}\Big)^{-\alpha_2}Z_{\alpha_1,-\beta_1,-\alpha_2,\beta_2,qh,k}(t,1,1,\chi_D,\chi_D)\bigg]dt+O_\varepsilon\left(T^{1-\varepsilon}\right)\nonumber
\end{align*}
uniformly for $\alpha_i,\beta_i\ll\ml^{-1}$, where $\varepsilon_{\chi_D}$ is the root factor in the functional equation for $L(s,\chi_D)$.
\end{thm}

Theorem \ref{twists} gives
\[I(\ul{\alpha},\ul{\beta})=I_1+I_2+I_3+I_4+I^\prime_1+I^\prime_2+O_\varepsilon\left(T^{1-\varepsilon}\right),\]
say, where the terms $I_1, I_2, I_3, I_4$ correspond to the first four terms of the (expanded) above expression and $I^\prime_1, I^\prime_2$  correspond to the last two. We shall show in Section \ref{lower order sec} that these latter terms are of a lower order.

\subsection{Main terms}\label{main term sec}Consider the first term
\begin{align*}
I_1=&\sum_{h_1h_2,k_1k_2\leq y}\frac{d_r(h_1)d_r(h_2)\chi_D(h_2)P[h_1h_2]d_r(k_1)d_r(k_2)\chi_D(k_2)P[k_1k_2]}{(h_1h_2)^{1/2+\alpha_3}(k_1k_2)^{1/2+\beta_3}}\\
&\qquad\quad\times\sum_{h_1h_2m_1m_2=k_1k_2n_1n_2}\frac{\chi_D(m_2)\chi_D(n_2)}{m_1^{1/2+\alpha_1}m_2^{1/2+\alpha_2}n_1^{1/2+\beta_1}n_2^{1/2+\beta_2}}\int_{-\infty}^{\infty}\Phi\Big(\frac{t}{T}\Big)V\Big(\frac{m_1m_2n_1n_2}{t^2}\Big)dt.
\end{align*}
The terms $I_2, I_3, I_4$ can be acquired by permuting the shifts and by multiplying by the appropriate factors of $T$ and $q$. We therefore concentrate on $I_1$ in the meanwhile. \\

The contour integral representation \eqref{P integral} and \eqref{Vx} give
\begin{multline*}I_1=\sum_{i,j}\frac{b_ib_ji!j!}{(\log y)^{i+j}}\Big(\frac{1}{2\pi i}\Big)^3\int_{-\infty}^{\infty}\int_{(1)}\int_{(1)}\int_{(1)}\Phi\Big(\frac{t}{T}\Big)G(s)\Big(\frac{t}{2\pi}\Big)^{2s} y^{u+v}\\
\times\sum_{h_1h_2m_1m_2=k_1k_2n_1n_2}\frac{d_r(h_1)d_r(h_2)\chi_D(h_2)d_r(k_1)d_r(k_2)\chi_D(k_2)}{(h_1h_2)^{1/2+\alpha_3+u}(k_1k_2)^{1/2+\beta_3+v}}\\
\times\frac{\chi_D(m_2)\chi_D(n_2)}{m_1^{1/2+\alpha_1+s}m_2^{1/2+\alpha_2+s}n_1^{1/2+\beta_1+s}n_2^{1/2+\beta_2+s}}\frac{du}{u^{i+1}}\frac{dv}{v^{j+1}}\frac{ds}{s}dt.
\end{multline*}
Since the coefficients and the condition $h_1h_2m_1m_2=k_1k_2n_1n_2$ are both multiplicative, we may express the inner sum in terms of an Euler product. From this we see that the inner sum is given by 
\begin{multline}\label{euler prod}A(\ul{\alpha},\ul{\beta},u,v)\zeta(1\!+\!\alpha_3\!+\!\beta_3\!+\!u\!+\!v)^{2r^2}\zeta(1\!+\!\alpha_3\!+\!\beta_1\!+\!u+\!s)^r
\zeta(1\!+\!\alpha_3\!+\!\beta_2\!+\!u+\!s)^r\\\times\zeta(1\!+\!\beta_3\!+\!\alpha_1\!+\!v+\!s)^r\zeta(1\!+\!\beta_3\!+\!\alpha_2\!+\!v+\!s)^r
\zeta(1\!+\!\alpha_1\!+\!\beta_1+\!2s)\zeta(1\!+\!\alpha_2\!+\!\beta_2+\!2s),
\end{multline}
where $A(\ul{\alpha},\ul{\beta},u,v)$ is some arithmetic factor which is absolutely convergent in the product of the half-planes $\Re(u)>-1/2,\, \Re(v)>-1/2$.

We first move the $u$ and $v$ contours to $\Re(u)=\Re(v)=\delta$, and then move the $s$ contour to $\Re(s)=-\delta/2$, where
$\delta> 0$ is some fixed constant such that the arithmetic factor converges absolutely. In doing so, we only cross a pole at $s=0$. Note that the poles at $s=-(\alpha_i+\beta_i)/2$ for $i=1,2$ of the zeta-functions are cancelled out by the zeros of $G(s)$. On the new line of integration we simply bound the integral by absolute values, giving the following contribution
\[
\ll_\varepsilon T^{1+\varepsilon}y^{2\delta}T^{-\delta}\ll_\varepsilon T^{1-\varepsilon}.
\] Therefore, 
\begin{equation}\label{I_1}I_1=\widehat{\Phi}(0)T\zeta(1\!+\!\alpha_1\!+\!\beta_1)\zeta(1\!+\!\alpha_2\!+\!\beta_2)\sum_{i,j}\frac{b_ib_ji!j!}{(\log y)^{i+j}}J_{i,j}+O_\varepsilon(T^{1-\varepsilon}),\end{equation}
where
\begin{equation*}
\begin{split}
J_{i,j}=&\Big(\frac{1}{2\pi i}\Big)^2\int_{(1)}\int_{(1)} y^{u+v}A(\ul{\alpha},\ul{\beta},u,v)\zeta(1\!+\!\alpha_3\!+\!\beta_3\!+\!u\!+\!v)^{2r^2}\zeta(1\!+\!\alpha_3\!+\!\beta_1\!+\!u)^r\\
&\qquad\times\zeta(1\!+\!\alpha_3\!+\!\beta_2\!+\!u)^r\zeta(1\!+\!\beta_3\!+\!\alpha_1\!+\!v)^r\zeta(1\!+\!\beta_3\!+\!\alpha_2\!+\!v)^r\frac{du}{u^{i+1}}\frac{dv}{v^{j+1}}\\
=&\sum_{n\leq y}\frac{d_{2r^2}(n)}{n^{1+\alpha_3+\beta_3}}\Big(\frac{1}{2\pi i}\Big)^2\int_{(1)}\int_{(1)} \Big(\frac{y}{n}\Big)^{u+v}A(\ul{\alpha},\ul{\beta},u,v)\zeta(1\!+\!\alpha_3\!+\!\beta_1\!+\!u)^r\\
&\qquad\times\zeta(1\!+\!\alpha_3\!+\!\beta_2\!+\!u)^r\zeta(1\!+\!\beta_3\!+\!\alpha_1\!+\!v)^r\zeta(1\!+\!\beta_3\!+\!\alpha_2\!+\!v)^r\frac{du}{u^{i+1}}\frac{dv}{v^{j+1}}.
\end{split}
\end{equation*}
Here we have restricted the sum to  $n\leq y$ since the error term can be made arbitrarily small by moving the contours of integration sufficiently far enough to the right. 

We now shift the contours of integration to $\Re(u)=\Re(v)= \ml^{-1}$. A trivial estimate then gives $J_{i,j}\ll \ml^{i+j+2r^2+4r}$. Since $A(\ul{\alpha},\ul{\beta},u,v)$ is holomorphic at $(\ul{0},\ul{0},0,0)$ we have the Taylor expansion
\[
A(\ul{\alpha},\ul{\beta},u,v)=A(\ul{0},\ul{0},0,0)+O(\ml^{-1})+O(|u|+|v|).
\]
So, upon replacing $A(\ul{\alpha},\ul{\beta},u,v)$ by $A(\ul{0},\ul{0},0,0)$, we incur an error of size $O(\ml^{i+j+2r^2+4r-1})$. 

We take a brief diversion to show that $A(\ul{0},\ul{0},0,0)=A_{r+1}$. We have 
\begin{equation*}
\begin{split}
A(\ul{s},\ul{s},0,0)=&\zeta(1\!+\!s)^{-2(r+1)^2}\!\!\!\!\!\!\sum_{\substack{h_1h_2m_1m_2=\\k_1k_2n_1n_2}}\!\!\!\!\!\!\!\frac{d_r(h_1)d_r(h_2)\chi_D(h_2)d_r(k_1)d_r(k_2)\chi_D(k_2)\chi_D(m_2)\chi_D(n_2)}{(h_1h_2m_1m_2 k_1k_2n_1n_2)^{1/2+s}}\\
=&\zeta(1\!+\!s)^{-2(r+1)^2}\sum_{m=n}\frac{a_{r+1}(m)a_{r+1}(n)}{(mn)^{1/2+s}}\\
=&\zeta(1\!+\!s)^{-2(r+1)^2}\sum_{n=1}^\infty\frac{a_{r+1}(n)^2}{n^{1+2s}}.
\end{split}
\end{equation*}
It is not immediately obvious that this is convergent at $s\!=\!0$. However, we may arrive at such a conclusion by first replacing $\zeta(1+s)^{-2(r+1)^2}$ with $[\zeta_K(1+s)/L(1+s,\chi_D)]^{-2(r+1)^2}$, forming the Euler product of the zeta-function and the sum over $n$, and then separately checking the cases of split, inert and ramified primes (for which one has $\chi(p)=1$, $\chi(p)=-1$ and $\chi(p)=0$ respectively). Setting $s\!=\!0$ then gives $A_{r+1}$.

On returning to our formula for $J_{i,j}$, we now have
\[J_{i,j}=A_{r+1}\sum_{n\leq y}\frac{d_{2r^2}(n)}{n^{1+\alpha_3+\beta_3}}K_i(\alpha_3+\beta_1,\alpha_3+\beta_2)K_j(\beta_3\!+\alpha_1,\beta_3\!+\alpha_2)+O(\ml^{i+j+2r^2+4r-1}),\]
where
\[K_j(\alpha,\beta)=\frac{1}{2\pi i}\int_{(\ml^{-1})}\Big(\frac{y}{n}\Big)^{u}\zeta(1\!+\!\alpha\!+\!u)^r\zeta(1\!+\!\beta\!+\!u)^r\frac{du}{u^{j+1}}.\]
By Lemmas \ref{K lem} and \ref{d sum lem}, we have
\begin{align*}
J_{i,j}&=\frac{A_{r+1}(\log y)^{i+j+4r}}{((r-1)!)^4i!j!}\ \iiiint\limits_{\substack{0\leq x_1,x_2,x_3,x_4\leq 1\\x_1+x_2,x_3+x_4\leq 1}}(x_1x_2x_3x_4)^{r-1}(1\!-\!x_1\!-\!x_2)^i(1\!-\!x_3\!-\!x_4)^j\\
&\qquad\times\sum_{n\leq y}\frac{d_{2r^2}(n)}{n^{1+\alpha_3+\beta_3}}\Big(\frac{y}{n}\Big)^{-(\alpha_3+\beta_1) x_1-(\alpha_3+\beta_2) x_2-(\beta_3+\alpha_1)x_3-(\beta_3+\alpha_2)x_4}\\
&\qquad \qquad\times\Big(\frac{\log y/n}{\log y}\Big)^{i+j+4r}dx_1dx_2dx_3dx_4+O(\ml^{i+j+2r^2+4r-1})\\
&=\frac{A_{r+1}(\log y)^{i+j+2r^2+4r}}{(2r^2-1)!((r-1)!)^4i!j!}\int_0^1\iiiint\limits_{\substack{0\leq x_1,x_2,x_3,x_4\leq 1\\x_1+x_2,x_3+x_4\leq 1}}(x_1x_2x_3x_4)^{r-1}(1\!-\!x_1\!-\!x_2)^i\\
&\qquad \times(1\!-\!x_3\!-\!x_4)^jy^{-(\alpha_3+\beta_3)x-\big((\alpha_3+\beta_1) x_1+(\alpha_3+\beta_2) x_2+(\beta_3+\alpha_1)x_3+(\beta_3+\alpha_2)x_4\big)(1-x)}\\
&\qquad\qquad \times x^{2r^2-1}(1\!-\!x)^{i+j+4r}dx_1dx_2dx_3dx_4dx+O(\ml^{i+j+2r^2+4r-1})\\
&=\frac{A_{r+1}(\log y)^{i+j+2r^2+4r}}{(2r^2-1)!((r-1)!)^4i!j!}\qquad\int\!\!\!\int\!\!\!\!\!\!\!\!\!\!\!\!\!\!\!\!\!\!\!\!\!\int\limits_{\substack{0\leq x, x_1,x_2,x_3,x_4\leq 1\\x+x_1+x_2,x+x_3+x_4\leq 1}}\!\!\!\!\!\!\!\!\!\!\!\!\!\!\!\!\!\!\!\!\!\int\!\!\!\int(x_1x_2x_3x_4)^{r-1}(1\!-\!x\!-\!x_1\!-\!x_2)^i\\
&\qquad \times(1\!-\!x\!-\!x_3\!-\!x_4)^j y^{-(\alpha_3+\beta_3)x-(\alpha_3+\beta_1) x_1-(\alpha_3+\beta_2) x_2-(\beta_3+\alpha_1)x_3-(\beta_3+\alpha_2)x_4}\\
&\qquad\qquad \times x^{2r^2-1}dx_1dx_2dx_3dx_4dx+O(\ml^{i+j+2r^2+4r-1}),
\end{align*}
where in the last step we have used the substitutions $x_m(1-x)\mapsto x_m$, $m=1,2,3,4$. 
Inputting this in our formula for $I_1$  gives 
\begin{align*}I_1=&\ \widehat{\Phi}(0)T\frac{A_{r+1}(\log y)^{2r^2+4r}}{(2r^2-1)!((r-1)!)^4}\zeta(1\!+\!\alpha_1\!+\!\beta_1)\zeta(1\!+\!\alpha_2\!+\!\beta_2)\\
&\qquad\times\int\!\!\!\int\!\!\!\!\!\!\!\!\!\!\!\!\!\!\!\!\!\!\!\!\!\int\limits_{\substack{0\leq x, x_1,x_2,x_3,x_4\leq 1\\x+x_1+x_2,x+x_3+x_4\leq 1}}\!\!\!\!\!\!\!\!\!\!\!\!\!\!\!\!\!\!\!\!\!\int\!\!\!\int y^{-(\alpha_3+\beta_3)x-(\alpha_3+\beta_1) x_1-(\alpha_3+\beta_2) x_2-(\beta_3+\alpha_1)x_3-(\beta_3+\alpha_2)x_4}x^{2r^2-1}(x_1x_2x_3x_4)^{r-1}\\
&\qquad\qquad\qquad\times P(1\!-\!x\!-\!x_1\!-\!x_2)P(1\!-\!x\!-\!x_3\!-\!x_4)dxdx_1dx_2dx_3dx_4+O(T\ml^{2r^2+4r+1}).
\end{align*}

We now recall the formulae for the  three remaining $I$ terms;
\begin{align*}
I_2=&T^{-(\alpha_1+\beta_1)}I_1(-\beta_1,\alpha_2,-\alpha_1,\beta_2)+O(T\ml^{2r^2+4r+1}),\\
I_3=&(qT)^{-(\alpha_2+\beta_2)}I_1(\alpha_1,-\beta_2,\beta_1,-\alpha_2)+O(T\ml^{2r^2+4r+1}),\\
I_4=&T^{-(\alpha_1+\beta_1)}(qT)^{-(\alpha_2+\beta_2)}I_1(-\beta_1,-\beta_2,-\alpha_1,-\alpha_2)+O(T\ml^{2r^2+4r+1}).
\end{align*}
These give
\begin{align*}I(\ul{\alpha},\ul{\beta})=&\ 
\widehat{\Phi}(0)T\frac{A_{r+1}(\log y)^{2r^2+4r}}{(2r^2-1)!((r-1)!)^4}
\,\,\,\,\int\!\!\int\!\!\!\!\!\!\!\!\!\!\!\!\!\!\!\!\!\!\!\!\int\limits_{\substack{0\leq x, x_1,x_2,x_3,x_4\leq 1\\x+x_1+x_2,x+x_3+x_4\leq 1}}\!\!\!\!\!\!\!\!\!\!\!\!\!\!\!\!\!\!\!\!\int\!\!\int\ \  y^{-(\alpha_3+\beta_3)x-\alpha_3(x_1+x_2)-\beta_3(x_3+x_4)}\\
&\qquad \times U(\ul{x})x^{2r^2-1}(x_1x_2x_3x_4)^{r-1}P(1\!-\!x\!-\!x_1\!-\!x_2)P(1\!-\!x\!-\!x_3\!-\!x_4)
dxdx_1dx_2dx_3dx_4\\
&\qquad\qquad\quad +I^\prime_1+I^\prime_2+O(T\ml^{2r^2+4r+1})+O_\varepsilon(T^{1-\varepsilon}),
\end{align*}
where
\begin{equation*}
\begin{split}
U(\ul{x})=&\frac{y^{-\beta_1x_1-\beta_2x_2-\alpha_1x_3-\alpha_2x_4}}{(\alpha_1+\beta_1)(\alpha_2+\beta_2)}-\frac{T^{-(\alpha_1+\beta_1)}y^{\alpha_1x_1-\beta_2x_2+\beta_1x_3-\alpha_2x_4}}{(\alpha_1+\beta_1)(\alpha_2+\beta_2)}\\
&-\frac{(qT)^{-(\alpha_2+\beta_2)}y^{-\beta_1x_1+\alpha_2x_2-\alpha_1x_3+\beta_2x_4}}{(\alpha_1+\beta_1)(\alpha_2+\beta_2)}+\frac{T^{-(\alpha_1+\beta_1)}(qT)^{-(\alpha_2+\beta_2)}y^{\alpha_1x_1+\alpha_2x_2+\beta_1x_3+\beta_2x_4}}{(\alpha_1+\beta_1)(\alpha_2+\beta_2)}.
\end{split}
\end{equation*}
A short calculation shows that  $U(\ul{x})$ is given by
\begin{align*}
&y^{-\beta_1x_1-\beta_2x_2-\alpha_1x_3-\alpha_2x_4}\bigg(\frac{1-[Ty^{-(x_1+x_3)}]^{-(\alpha_1+\beta_1)}}{\alpha_1+\beta_1}\bigg)\bigg(\frac{1-[qTy^{-(x_2+x_4)}]^{-(\alpha_2+\beta_2)}}{\alpha_2+\beta_2}\bigg)\\
&\qquad=y^{-\beta_1x_1-\beta_2x_2-\alpha_1x_3-\alpha_2x_4}\ml^2\big(1-\vt(x_1+x_3)\big)\big(1-\vt(x_2+x_4)\big)\\
&\qquad\qquad\times\int_0^1\int_0^1 [Ty^{-(x_1+x_3)}]^{-(\alpha_1+\beta_1)t_1}[qTy^{-(x_2+x_4)}]^{-(\alpha_2+\beta_2)t_2}dt_1dt_2+O(\ml).
\end{align*}
Inputting this into $I(\ul{\alpha},\ul{\beta})$ gives the main term of Proposition \ref{prop:key}.

\subsection{Lower order terms}\label{lower order sec}
It remains to calculate $I^\prime_1, I^\prime_2$. We have
\begin{align*}I_1^\prime=&\varepsilon_{\chi_D}\sum_{h_1h_2,k_1k_2\leq y}\frac{d_r(h_1)d_r(h_2)\chi_D(h_2)P[h_1h_2]d_r(k_1)d_r(k_2)\chi_D(k_2)P[k_1k_2]}{(h_1h_2)^{1/2+\alpha_3}(k_1k_2)^{1/2+\beta_3}}\\
&\qquad\qquad\times\sum_{h_1h_2m_1m_2=qk_1k_2n_1n_2}\frac{\chi_D(m_2)\chi_D(m_2)}{m_1^{1/2-\beta_2}m_2^{1/2+\alpha_2}n_1^{1/2+\beta_1}n_2^{1/2-\alpha_1}}\\
&\qquad\qquad\qquad\qquad\int_{-\infty}^{\infty}\Phi\Big(\frac{t}{T}\Big)\Big(\frac{t}{2\pi}\Big)^{-\alpha_1}\Big(\frac{qt}{2\pi}\Big)^{-\beta_2}V\Big(\frac{m_1m_2n_1n_2}{t^2}\Big)dt.
\end{align*}
Similarly to before we may unfold the integrals in the polynomials $P[h_1h_2], P[k_1k_2]$ and use \eqref{Vx}. We then encounter an inner sum of the form  
\begin{align*}&\sum_{h_1h_2m_1m_2=qk_1k_2n_1n_2}\frac{d_r(h_1)d_r(h_2)\chi_D(h_2)d_r(k_1)d_r(k_2)\chi_D(k_2)}{(h_1h_2)^{1/2+\alpha_3}(k_1k_2)^{1/2+\beta_3}}\\
&\qquad\qquad\qquad\qquad\times \frac{\chi_D(m_2)\chi_D(m_2)}{m_1^{1/2-\beta_2+s}m_2^{1/2+\alpha_2+s}n_1^{1/2+\beta_1+s}n_2^{1/2-\alpha_1+s}}.
\end{align*}
From the Euler product this is given by
\begin{multline*}B(\ul{\alpha},\ul{\beta},u,v)\zeta(1\!+\!\alpha_3\!+\!\beta_3\!+\!u\!+\!v)^{2r^2}\zeta(1\!+\!\alpha_3\!-\!\alpha_1\!+\!u+\!s)^r\\
\times\zeta(1\!+\!\alpha_3\!+\!\beta_1\!+\!u+\!s)^r\zeta(1\!+\!\beta_3\!-\!\beta_1\!+\!v+\!s)^r\zeta(1\!+\!\beta_3\!+\!\alpha_2\!+\!v+\!s)^r,
\end{multline*}
where $B(\ul{\alpha},\ul{\beta},u,v)$ is an arithmetic factor convergent in some product of half-planes containing the origin. In comparison with \eqref{euler prod}, note the absence of the two zeta terms.
Accounting for this whilst shifting the contours of integration to $\Re(u)=\Re(v)= \ml^{-1}$ and trivially evaluating the integrals as in the analysis of the previous section gives 
\[I_1^\prime\ll (\log y)^{2r^2+4r}.\]

The same bound holds for $I_2^\prime$. 

\section{Acknowledgements}
The authors thank the referee for helpful comments and suggestions.

\section{Funding}
This work is supported by the Leverhulme Trust [grant number RPG-049 to H.M.B.]; the Research Council of Norway [grant number 227768 to W.H.]; and by an Association for Women in Mathematics Travel Grant, funded through the Nation Science Foundation [grant number DMS-1153905 to C.L.T.].

\end{document}